\documentclass[10pt]{amsart}
\usepackage{amsmath}
\usepackage{amssymb}
\usepackage{amsfonts}
\usepackage{amsthm}
\usepackage{enumerate}
\usepackage[all]{xy}
\usepackage[latin1]{inputenc}
\usepackage{mathdots}
\usepackage{graphicx}
\usepackage{latexsym}
\usepackage{mathrsfs}

\usepackage{tikz}

\PassOptionsToPackage{usenames,dvipsnames,svgnames}{xcolor}

\input xy
\def\st{\,\mid\,}
\newcommand{\bigopls}{\mathop{\mathchoice
  {\textstyle\bigoplus}{\textstyle\bigoplus}
  {\scriptstyle\bigoplus}{\scriptscriptstyle\bigoplus}}}
\newtheorem{Theo}{Theorem}[section]
\newtheorem*{Theor}{Theorem}
\newtheorem{Prop}[Theo]{Proposition}
\newtheorem{Cor}[Theo]{Corollary}
\newtheorem{Lemma}[Theo]{Lemma}

\newtheorem{Conj}[Theo]{Conjecture}
\newtheorem*{Conjec}{Conjecture}
\theoremstyle{definition}

\newtheorem{Remark}[Theo]{Remark}

\def\mystrut(#1,#2){\vrule height #1pt depth #2pt width 0pt}

\newcommand{\Hom}{{\rm Hom}}
\newcommand{\Ext}{{\rm Ext}}

\newcommand{\mmod}{{\rm mod}}

\newcommand{\Z}{\mathbb{Z}}
\newcommand{\C}{\mathcal{C}}

\newcommand{\GL}{{\rm GL}}

\def\s{\sigma}
\def\t{\tau}
\def\z{\zeta}

\begin{document}

\title[Homological dimensions of subalgebras]{Homological dimensions for co-rank one idempotent subalgebras}

\author{Colin Ingalls}

\author{Charles Paquette}

\maketitle

\begin{abstract} Let $k$ be an algebraically closed field and $A$ be a (left and right) Noetherian associative $k$-algebra. Assume further that $A$ is either positively graded or semiperfect (this includes the class of finite dimensional $k$-algebras, and $k$-algebras that are finitely generated modules over a Noetherian central Henselian ring). Let $e$ be a primitive idempotent of $A$, which we assume is of degree $0$ if $A$ is positively graded. We consider the idempotent subalgebra $\Gamma = (1-e)A(1-e)$ and $S_e$ the simple right $A$-module $S_e = eA/e{\rm rad}A$, where ${\rm rad}A$ is the Jacobson radical of $A$, or the graded Jacobson radical of $A$ if $A$ is positively graded. In this paper, we relate the homological dimensions of $A$ and $\Gamma$, using the homological properties of $S_e$. First, if $S_e$ has no self-extensions of any degree, then the global dimension of $A$ is finite if and only if that of $\Gamma$ is. On the other hand, if the global dimensions of both $A$ and $\Gamma$ are finite, then $S_e$ cannot have self-extensions of degree greater than one, provided $A/{\rm rad}A$ is finite dimensional.
\end{abstract}

\section{Introduction}

Let $k$ be an algebraically closed field and $A$ an associative $k$-algebra. The left (or right) global dimension of $A$ is an nice homological invariant of the algebra which is defined as the supremum of the projective dimensions of all left (resp. right) $A$-modules. The notion of global dimension, which traces back to Cartan and Eilenberg in \cite{CartanEilenberg}, has been widely studied in the past decades. It was first studied in the context of commutative algebras. When $A$ is commutative Noetherian, Auslander and Buchsbaum have proven the well known fact that $A$ is regular if and only if its global dimension is finite. This result has been proven independently by Serre in \cite{Serre}. In \cite{Aus}, Auslander gave many interesting properties of the global dimension in the context of noncommutative algebras. One of his key results shows that when $A$ is both left and right Noetherian, then the left global dimension of $A$ coincides with the right global dimension of $A$. In this paper, an algebra which is both left and right Noetherian is called \emph{Noetherian}, and in this case, the global dimension of $A$ is denoted ${\rm gl.dim}\,A$. In the representation theory of finite dimensional algebras, Happel \cite{Happel} has shown the very important fact that the bounded derived category $D^b({\rm mod}\, A)$ of the category ${\rm mod}\, A$ of finite dimensional right $A$-modules admits a Serre functor if and only if ${\rm gl.dim}\,A$ is finite. In particular, $D^b({\rm mod}\, A)$ admits Auslander-Reiten triangles if and only if ${\rm gl.dim}\,A$ is finite. This also shows that the finiteness of the global dimension is a derived invariant. The global dimension also has applications in noncommutative algebraic geometry. Namely, if $\mathbb{X}$ is a projective variety, then one can consider the bounded derived category $D^b({\rm coh}(\mathbb{X}))$ of the coherent sheaves over $\mathbb{X}$; and this derived category admits a Serre functor if $\mathbb{X}$ is smooth. When $D^b({\rm coh}(\mathbb{X}))$ admits a tilting object $T$, then $D^b({\rm coh}(\mathbb{X}))$ is triangle equivalent to $D^b({\rm mod}\, {\rm End}(T))$, where ${\rm End}(T)$ is a finite dimensional $k$-algebra. Hence in this case, if $\mathbb{X}$ is smooth, then ${\rm End}(T)$ has finite global dimension.

\medskip

In order to compute the global dimension of an algebra $A$, it is often easier to reduce the computations to a smaller algebra. One way to reduce the size of an algebra is to find another closely related algebra with fewer simple modules. In this paper, we work in the wide context of associative Noetherian $k$-algebras. Our algebras are not assumed to be commutative. All algebras in this paper are assumed to be (associative) Noetherian $k$-algebras, unless otherwise indicated. In most of the results, the algebras considered are either semiperfect or positively graded.
The definitions of a semiperfect algebra and a positively graded algebra are recalled in the corresponding sections. Semiperfect $k$-algebras have finitely many - say $n$ - non-isomorphic simple right $A$-modules, and the identity $1_A$ of $A$ decomposes as a finite sum $1_A = e_1 + \cdots + e_n$ of pairwise orthogonal idempotents, each such idempotent corresponds to a unique isomorphism class of simple right $A$-modules. Let $e$ be a fixed nonzero idempotent of $A$ with $S_e$ the semi-simple top of $eA$, and let $\Gamma:=(1-e)A(1-e)$ be the corresponding idempotent subalgebra. The algebra $\Gamma$ is again Noetherian and semiperfect and admits fewer simple $\Gamma$-modules, up to isomorphism. A similar behavior arises for graded simple modules over a positively graded algebra. Consequently, we can define $e, S_e$ and $\Gamma$ as follows. A positively graded $k$-algebra $A$ has finitely many - say $n$ - non-isomorphic graded simple right $A$-modules, up to a shift, and the identity $1_A$ of $A$ decomposes as a finite sum $1_A = e_1 + \cdots + e_n$ of pairwise orthogonal idempotents of degree $0$, each such idempotent corresponds to a unique isomorphism class of graded simple right $A$-modules of degree $0$. Let $e$ be a fixed nonzero idempotent of degree $0$ of $A$ with $S_e$ the semi-simple graded top of $eA$, and let $\Gamma:=(1-e)A(1-e)$ be the corresponding idempotent subalgebra. The algebra $\Gamma$ is again Noetherian and positively graded and admits fewer graded simple $\Gamma$-modules up to isomorphism and shift.

\medskip

Even if the algebras $A, \Gamma$ are closely related, their homological behaviors can be very different. One can easily find examples where $A$ has finite global dimension while $\Gamma$ does not, and conversely. However, the homological properties of the semi-simple module $S_e$ gives more information on the relationship between the global dimensions of $A$ and $\Gamma$. When $S_e$ is simple, that is, when $e$ is primitive, we will show that this relationship is much stronger. We will show the following theorem, for when $A$ is semiperfect or positively graded.

\begin{Theor}
Suppose that $\Ext_A^i(S_e, S_e)=0$ for $i > 0$.  Then the global dimension of $A$ is finite if and only if the global dimension of $\Gamma$ is finite. More precisely, we have bounds
$${\rm gl.dim}\,\Gamma \le {\rm max}({\rm id}_A S_e + {\rm pd}_A S_e -1, {\rm gl.dim}\, A)$$ and $${\rm gl.dim}\,A \le 2 {\rm gl.dim}\,\Gamma + 2.$$
\end{Theor}
In the statement, id$_A$ stands for the injective dimension and pd$_A$ stands for the projective dimension.
Surprisingly, when $A$ is finite dimensional (hence is Noetherian semiperfect with $A/{\rm rad}A$ finite dimensional) and $e$ is primitive, if one assumes that both ${\rm gl.dim}\,A, {\rm gl.dim}\,\Gamma$ are finite, then $\Ext_A^i(S_e, S_e)=0$ for $i > 0$. We actually get a stronger version as follows.

\begin{Theor}
Assume that $A$ is finite dimensional, $e$ is primitive and both ${\rm pd}_A S_e$ and ${\rm pd}_\Gamma (eA(1-e))$ are finite. Then $\Ext_A^i(S_e, S_e)=0$ for $i > 0$.
\end{Theor}

If $A$ is not finite dimensional, the above theorem is clearly not true. Take for instance the one-point extension $$A = \left(%
\begin{array}{cc}
  k[[x]] & k[[x]]/\langle x \rangle\\
  0 & k \\
\end{array}%
\right),$$ where $k[[x]]$ is the $k$-algebra of formal power series in one variable. Then $A$ has global dimension two and is Noetherian semiperfect. For the idempotent $e=e_{11}$ of $A$, the simple module $S_e$ is not self-orthogonal since $\Ext^1_A(S_e, S_e) \ne 0$ and $\Gamma = e_{22}Ae_{22} \cong k$ has global dimension zero. If one rather considers the polynomial algebra $k[x]$ instead of $k[[x]]$, then the analogues of $A,\Gamma$ defined above are Noetherian positively graded and for the same reason, we get a counter-example. Observe however that when $A$ is finite dimensional and $e$ is primitive with ${\rm pd}_A S_e < \infty$, then the condition $\Ext^1_A(S_e,S_e)=0$ is automatically verified. This indeed follows from \cite{ILP}. So in the general case ($A$ is a Noetherian $k$-algebra and is either semiperfect or positively graded), we have the following conjecture.

\begin{Conjec}
Assume that $e$ is primitive and of degree zero if $A$ is positively graded, and assume that $\Ext^1_A(S_e,S_e)=0$. If both ${\rm pd}_A S_e$ and ${\rm pd}_\Gamma (eA(1-e))$ are finite, then $\Ext_A^i(S_e, S_e)=0$ for $i > 0$.
\end{Conjec}

In this paper, we show that the conjecture holds with the additional assumption that $A/{\rm rad}A$ is finite dimensional (which is verified in the positively graded case).
%in the following cases. Firstly, with the only assumption that $A$ is (Noetherian and) semiperfect, we prove that the conjecture holds when $A/{\rm rad}A$ is %finite dimensional. Secondly, we show that the conjecture holds when $A$ is (Noetherian and) positively graded.

\medskip

Finally, note that if $e$ is not primitive, then the above conjecture does not hold. For instance, take the algebra $A = kQ/I$, where $Q$ is the quiver
$$
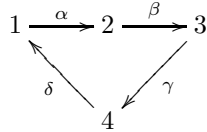
\xymatrix{1 \ar[r]^\alpha & 2 \ar[r]^\beta & 3 \ar[dl]^\gamma \\ & 4 \ar[ul]^\delta}$$
with $I = \langle \delta\gamma, \alpha\delta \rangle$ and take $e$ the sum of the primitive idempotents of the vertices $1,3$. The algebra $A$ has global dimension three while $\Gamma$ is hereditary. Clearly, $\Ext^1_A(S_e,S_e)=0$. However, $\Ext^2_A(S_e, S_e) \ne 0$ because of the minimal relation $\delta\gamma$. In general, we do not know whether there exists a right $A$-module $M_e$ that completely controls the relationship between the homological dimensions of $A$ and $\Gamma$, when $e$ is not primitive.

\section{Semiperfect Noetherian algebras}

We refer the reader  to \cite[page 301]{AndFul} for properties of semiperfect algebras. Let $A$ be an associative $k$-algebra where $k$ is algebraically closed. We denote by mod$\,A$ the category of finitely generated right $A$-modules. Let ${\rm rad} A$ denote the Jacobson radical of $A$, that is, the intersection of all maximal right (or left) ideals of $A$. Then $A$ is \emph{semiperfect} if $A/{\rm rad}A$ is a semi-simple $k$-algebra, and idempotents lift modulo ${\rm rad} A$. By the well known Wedderburn-Artin theorem, the first condition means
$$A/{\rm rad}A \cong M_{m_1}(k_1) \times \cdots \times M_{m_n}(k_n)$$
as $k$-algebras, where for $1 \le i \le n$, $M_{m_i}(k_i)$ is the simple $k$-algebra of all $m_i \times m_i$ matrices over a division $k$-algebra $k_i$. If $A$ (or $A/{\rm rad}A$) is finite dimensional, since $k = \bar k$, we have $k_i = k$ for all $i$. Using the lifting of idempotents property, this yields a decomposition $1_A = e_1 + \cdots + e_n$ of $1_A$ into pairwise orthogonal idempotents. Note that the simple right $A$-modules are the simple right $A/{\rm rad}A$-modules and hence, there are exactly $n$ simple right $A$-modules up to isomorphism. In this section, every $k$-algebra considered is semiperfect, unless otherwise indicated.

\medskip

While semiperfect algebras form a nice class of algebras having well behaved homological properties, finitely generated modules may not have finitely generated projective resolutions. Even worse, for an arbitrary semiperfect algebra, the left global dimension may differ from the right global dimension. To avoid these problems, and since most of our applications fall in this class, we consider only Noetherian algebras, that is, algebras that are both left and right Noetherian. So in this section, all algebras considered are both semiperfect and Noetherian, unless otherwise indicated.

\medskip

Let $e$ be a fixed idempotent of $A$. By considering the \emph{idempotent subalgebra} $\Gamma:=(1-e)A(1-e)$, we are reducing the number of simple modules of the algebra, but keeping the property of $\Gamma$ being semiperfect and Noetherian; see \cite[Cor. 27.7]{AndFul} and \cite[Prop. 2.3]{Sando}. Since we are studying homological properties of algebras, and since the properties of being semiperfect and Noetherian are preserved under Morita-equivalence, for the same reason as above, we may assume that our algebra is basic, which means that $m_1 = \cdots = m_n=1$ and all the $e_j$ are primitive idempotents. For simplifying notation, there is no loss of generality in fixing $e = e_1$, when we are given that $e$ is primitive.
Observe that $e_1A, \ldots, e_nA$ represent all the indecomposable projective right $A$-modules, up to isomorphism; see \cite{AndFul}. In particular, every indecomposable projective module is cyclic.
There is an $A$-module which is of special interest for relating the homological properties of $A$ and $\Gamma$. This module is the semi-simple right $A$-module at $e$, that is, $S_e:=eA/e{\rm rad}A$.

\medskip

A class of examples of semiperfect Noetherian $k$-algebras are the finite dimensional $k$-algebras, which we know are Morita equivalent to $kQ/I$ for some finite quiver $Q$ and some admissible ideal $I$ of $kQ$.  Other examples are obtained as follows.

Let $Q$ be a finite quiver and denote by $J_Q$ the ideal of $kQ$ generated by all arrows. Let $I$ be an ideal of $kQ$ with $I \subseteq J_Q^2$ and which is generated by homogeneous elements. Let $\Lambda:=kQ/I$ and consider $J$ the ideal of $\Lambda$ generated by all classes of arrows. We can define a topology on $\Lambda$, which is called the \emph{$J$-adic topology}, as follows. A subset $U$ of $\Lambda$ is \emph{open} if for every $x$ in $U$, there exists an integer $r$ with $x + J^r \subseteq U$. It can be checked that this is a topology on $\Lambda$ such that the ring operations $$\cdot \,: \Lambda \times \Lambda \to \Lambda \;\;\; \text{and}\;\;\; +: \Lambda \times \Lambda \to \Lambda$$ are continuous, where we use the product topology on $\Lambda \times \Lambda$. This makes $\Lambda$ into a \emph{topological algebra}. The notion of a topological algebra is a classical notion that is widely studied by commutative algebraists. In noncommutative algebra, topological algebras are also studied, however, they do not share all the properties that hold for their commutative counterparts. The reader is invited to see, for instance, the work of Gabriel \cite{Gabriel}, where the notion of pseudo-compact algebras, which is a particular class of topological algebras, is used. We refer the reader to \cite[Chapter 10]{McDonald} for the very basic definitions and properties of topological algebras.

We say that $\Lambda$ is \emph{complete} is it is complete as a topological space, that is, every Cauchy sequence $(x_i)_{i \ge 0}$ in $\Lambda$ converges. Recall that a Cauchy sequence $(x_i)_{i \ge 0}$ in $\Lambda$ converges if for every neighborhood of zero $U$, there exists $r \ge 0$ such that for $s_1, s_2 \ge r$, we have $x_{s_1} - x_{s_2} \in U$. Since the powers of $J$ form a basis of neighborhoods of $0$, one may take a power of $J$ for the open set $U$. In our setting, the topology is always Hausdorff, since $\cap_{i\ge 1}J^i = 0$; see \cite[Lemma 10.1]{McDonald}. For $i \ge 0$, let $p_i : \Lambda/J^{i+1} \to \Lambda/J^i$ be the canonical projection.  The inverse limit $A:=\varprojlim \Lambda/J^i$ of the inverse system $$A/J \stackrel{p_1}{\longleftarrow} A/J^2 \stackrel{p_2}{\longleftarrow} \cdots $$ is an algebra and is called the \emph{completion} of $\Lambda$ with respect to the $J$-adic topology. Indeed, $A$ is a topological algebra with a $\hat J$-adic topology, where $\hat J$ is the ideal of $A$ defined as $\hat J : = \varprojlim J/J^i$. The algebra $A$ with the $\hat J$-adic topology is complete. We have an induced canonical map $\Lambda \to A$ whose kernel is $\cap_{i\ge 1}J^i = 0$. Therefore, $\Lambda$ can be viewed as a subalgebra of $A$.

\begin{Prop}
The completion $A$ of $\Lambda$ is semiperfect.
\end{Prop}

\begin{proof}
By \cite[Thm 27.6]{AndFul}, since $1_A = e_1 + \cdots + e_n$ is a decomposition of $1_A$ as a sum of pairwise orthogonal idempotents in $A$, it is sufficient to prove that for $1 \le i \le n$, $e_iAe_i$ is a local algebra. Fix $i$ with $1 \le i \le n$. Using the definition of $A$ as an inverse limit $A=\varprojlim \Lambda/J^i$, one may think of the algebra $e_iAe_i$ as follows. Consider $T'$ the set of all nonzero classes modulo $I$ of paths in $Q$ from $i$ to $i$. Observe that $T'$ contains a basis of $e_i \Lambda e_i$. So let $T\subseteq T'$ be a basis of $e_i \Lambda e_i$. Observe that we can define the length of an element in $T$ as the length of the corresponding path, and this is well defined since $I$ is homogeneous. Take a total order on the elements of $T$ refining path length, so that $T = \{t_0 = e_i, t_1, t_2, \ldots\}$. An element in $e_iAe_i$ can be thought of as a formal sum $\sum_{j \ge 0} \lambda_j t_j$ where $\lambda_j \in k$. Addition is done termwise and multiplication is done as for multiplying power series:
$$(\lambda_0 e_i + \lambda_1 t_1 + \lambda_2 t_2 + \cdots) \cdot (\mu_0 e_i + \mu_1 t_1 + \mu_2 t_2 + \cdots) = \sum_{t \in T}\left(\sum_{\stackrel{t_{i_1}, t_{i_2} \in T}{t_{i_1} t_{i_2} = t}}\lambda_{i_1}\mu_{i_2}\right)t.$$
It is then easily verified that an element $\sum_{j \ge 0} \lambda_j t_j$ has a left inverse if and only if $\lambda_0$ is non-zero. This proves, by Proposition 15.15 in \cite{AndFul}, that $e_iAe_i$ is a local algebra, and hence that $A$ is semiperfect.
\end{proof}

In general, the completion $A$ of an arbitrary topological Noetherian $k$-algebra $\Lambda$ may fail to be Noetherian. However, in our setting,  $A$ is always Noetherian when $\Lambda$ is, since $I$ is homogeneous; see \cite[Prop. 2.1]{Connell}.

\medskip

If $A$ is a finitely generated module over a commutative Noetherian Henselian local ring then $A$ is semiperfect.  Indeed, this characterizes Henselian rings as in \cite[Lemma 1.12.5]{ABAEM}.

Now, let us go back to the general theory, where $A$ is a fixed semiperfect Noetherian $k$-algebra. Using that $A$ is semiperfect, we get that the finitely generated projective $A$-modules satisfy the Krull-Schmidt decomposition theorem and hence, any indecomposable finitely generated projective $A$-module is isomorphic to some $e_iA$. Using this observation with the fact that $A$ is Noetherian gives that if $M \in {\rm mod}A$, then $M$ admits a projective resolution
$$\cdots \to P_r \to P_{r-1} \to \cdots \to P_1 \to P_0 \to M \to 0$$
such that for $i \ge 0$, each $P_i$ is a finite direct sum of copies of the modules $e_1A, \ldots, e_nA$.
Recall that since $A$ is Noetherian, the global dimension of $A$ is well defined and coincides with the left or right global dimension of $A$.  For a right $A$-module $M$, we denote by pd$_A M$ its projective dimension and by id$_A M$ its injective dimension. From \cite{Oso}, one has
$${\rm gl.dim}A = {\rm sup} \{{\rm pd}_A M \mid M \in {\rm mod} A\} = {\rm sup} \{{\rm id}_A M \mid M \in {\rm mod} A\},$$
which will be handy in the sequel. Also, if we know that the global dimension of $A$ is finite, then it coincides with the supremum of the projective dimensions of the simple right $A$-modules, see \cite{RobsonMcConnell}. However, we do not always know in advance that the global dimension is finite. In our setting, using the fact that our algebra is semiperfect, we have a stronger result.
\begin{Prop}
Let $A$ be Noetherian semiperfect. Then
$${\rm gl.dim}A = {\rm max}\left\{{\rm pd}_A \left(\frac{e_iA}{e_i{\rm rad}(A)}\right)\, \Big| \; i=1, \ldots, n\right\}.$$
\end{Prop}
\begin{proof}
Since $A$ is semi-local (that is, $A/{\rm rad}(A)$ is semi-simple), we can use Theorem $2$ in \cite{ChengXu}: the weak global dimension of $A$ is the flat dimension of the right $A$-module $A/{\rm rad}(A)$. Since $A/{\rm rad}(A)$ is finitely generated and $A$ is semi-perfect and right Noetherian, the flat dimension of the right $A$-module $A/{\rm rad}(A)$ coincides with its projective dimension. This implies that the weak global dimension of $A$ is the projective dimension of $A/{\rm rad}(A)$. Now, since $A$ is Noetherian, the weak global dimension coincide with the global dimension.
\end{proof}

\medskip

Recall that the \emph{radical} ${\rm rad}(M)$ of a right $A$-module $M$ is the intersection of all its maximal submodules. Since $A/{\rm rad}(A)$ is semi-simple, if $M$ is finitely generated, then ${\rm rad}(M)=M{\rm rad}(A)$. Recall also that a projective resolution
$$\cdots \stackrel{d_2}{\to} P_1 \stackrel{d_1}{\to} P_0 \to M \to 0$$
of $M \in \mmod(A)$ is \emph{minimal} if for $i \ge 1$, $d_i$ is a radical morphism, that is, the image of $d_i$ is contained in the radical of $P_{i-1}$. Every $M \in \mmod(A)$ admits a minimal projective resolution in $\mmod(A)$. Now, we have an exact functor
$$F:= \Hom((1-e)A,-): {\rm mod}A \to {\rm mod}\Gamma$$
between the corresponding categories of finitely generated right modules. Note that we also have a functor $G:=- \otimes_\Gamma (1-e)A : {\rm mod}\Gamma \to {\rm mod}A$ which is left adjoint to $F$. However, this functor is not exact. The reader is referred, for instance, to \cite{Psaroudakis} for a better idea of the canonical functors between the algebras $A, \Gamma$ and $A/A(1-e)A$. In this paper, we shall concentrate on the functor $F$. The following proposition collects some of the properties of the functor $F$.

\begin{Prop} \label{HomologicalProp}
Let $M, P, S \in {\rm mod} A$ with $P$ indecomposable projective and $S$ simple.
\begin{enumerate}[$(1)$]
    \item If $P$ is not isomorphic to a direct summand of $eA$, then $F(P)$ is indecomposable projective.
\item If $S$ is not isomorphic to a direct summand of $S_e$, then $F(S)$ is simple.
\item The functor $F$ is essentially surjective.
\item If $\Ext^i_A(M,S_e) = 0$ for all $i \ge 0$ and $\cdots \to P_1 \to P_0 \to M \to 0$ is a minimal projective resolution of $M$, then $\cdots \to F(P_1) \to F(P_0) \to F(M) \to 0$ is a minimal projective resolution of $F(M)$.
\item $F(S_e)=0$.
\end{enumerate}
\end{Prop}

\begin{proof}
Properties $(1),(2),(5)$ are easy and well known. For proving property $(3)$, it suffices to observe that $F$ and $G$ induce quasi-inverse equivalences between add$((1-e)A)$ and add$(\Gamma_\Gamma),$ the additive category generated by $\Gamma$ as a right $\Gamma$-module.
Hence, if $N$ is a finitely generated right $\Gamma$-module with a projective presentation $Q_1 \stackrel{f}{\rightarrow} Q_0 \to N \to 0$, then we get a projective presentation $G(Q_1) \stackrel{G(f)}{\rightarrow} G(Q_0) \to {\rm Coker}(G(f)) \to 0$. Then we see that $F({\rm Coker}(G(f))) \cong N$. Property $(4)$ follows from property $(1)$ by observing that if $g: P \to Q$ is a radical morphism with $P,Q \in {\rm add}((1-e)A)$, then $F(g)$ is a radical morphism.
\end{proof}

Therefore, if $M \in \mmod A$ is such that $\Ext_A^i(M,S_e) =0$ for all $i \ge 0$, then a minimal projective resolution of $F(M)$ is obtained by applying $F$ to a minimal projective resolution of $M$. Hence, in this case, pd$_A M = {\rm pd}_\Gamma F(M)$. In general, a module $M$  in $\mmod A$ needs not satisfy $\Ext_A^i(M,S_e) =0$ for all $i$, and we need to measure this defect. For $M \in \mmod A$, denote by $d_e(M)$ the maximal integer $i$ for which $\Ext^i_A(M,S_e)$ is non-zero (if $\Ext_A^i(M,S_e) =0$ for all $i \ge 0$, we set $d_e(M)=-1$; if $\Ext_A^i(M,S_e) \ne 0$ for infinitely many $i$, we set $d_e(M) = \infty$). The numbers $d_e(M)$ for $M \in \mmod A$ will be very handy in the sequel. Observe that the supremum of the $d_e(M)$ for $M \in {\rm Mod} A$ gives the injective dimension of $S_e$. However, using Baer's criteria, one only needs to take the supremum over the cyclic modules. In particular, we can take the supremum over finitely generated $A$ modules,
$${\rm id}_A S_e = {\rm sup} \{d_e(M) \mid M \in \mmod A\}.$$

\section{Noetherian positively graded $k$-algebras} \label{sectionNoethGraded}

In this section, $A$ is a Noetherian basic $k$-algebra which is \emph{positively graded} and generated in degrees $0$ and $1$. This means that
$$A = A_0 \oplus A_1 \oplus \cdots,$$
as $k$-vector spaces, where $A_0$ is a product of $n$ copies of $k$, and
$A_iA_j = A_{i+j}$ for all $i, j \ge 0$. We do not assume that $A$ is semiperfect. Observe that since $A$ is Noetherian, each $A_i$ should be finite dimensional. Examples of this include Noetherian $k$-algebras of the form $kQ/I$ where $Q$ is a finite quiver and $I$ is an ideal of $kQ$ generated by homogeneous elements of degree at least two.

\medskip

Observe that $1_A = e_1 + \cdots + e_n$ where the $e_i$ are primitive pairwise orthogonal idempotents of degree $0$. We let ${\rm rad}A$ denote the ideal $A_1 \oplus A_2 \oplus \cdots$ of $A$. This ideal does not necessarily coincide with the Jacobson radical of $A$, but we will see that it plays a similar role in the category of graded modules. For this reason, it is call the \emph{graded Jacobson radical} of $A$ and this is why we use that notation ${\rm rad}A$. Let $M$ be a right $A$-module.  One says that $M$ is \emph{graded} if $M$ admits a $k$-vector space decomposition
$$M = \bigopls_{i \in \Z}M_i$$
such that $M_jA_i \subseteq M_{i+j}$. Observe that for an idempotent $f$ in $A_0$, the projective module $fA$ is naturally graded, as
$$fA = fA_0 \oplus fA_1 \oplus fA_2 \oplus \cdots.$$
Let gr$A$ be the category of all finitely generated graded right $A$-modules. A \emph{morphism} $f: M \to N$ in gr$A$ is a morphism of $A$-modules such that $f(M_i) \subseteq f(N_i)$ for all $i$. Given $t \in \Z$ and $M \in {\rm gr}A$, we define $M[t]$ to be the module in gr$A$ such that $M[t]_i = M_{i-t}$. The right $A$-module $M[t]$ is called a \emph{shift} of $M$. The following essential facts about gr$A$ can all be found in \cite{MartinezSolberg}, for instance. In gr$A$, any indecomposable projective module is isomorphic to a shift of some $e_iA$. Given $f: L \to M$ a morphism in ${\rm gr}A$, we have that $f$ lies in the radical of the category gr$A$ if and only if $f(L) \subseteq M{\rm rad}A$. The graded submodule $M{\rm rad}A$ of $M$ is called the \emph{graded radical} of $M$. It coincides with the intersection of all graded submodules of $M$. The \emph{graded top} of $M$ is $M/M{\rm rad}A$. Moreover, every finitely generated module in gr$A$ has a projective cover in gr$A$. Using the Noetherian property, any $M \in {\rm gr}A$ admits a minimal finitely generated projective resolution
$$\cdots \to P_2 \to P_1 \to P_0 \to M \to 0$$
which is graded and such that each term $P_i$ is a finite direct sum of shifts of modules in $\{e_1A, \ldots, e_nA\}$.
For $M,N \in {\rm gr}A$ and $r \ge 0$, by $\Ext^r_A(M,N)$, we mean the $r$-th extension group of $M$ by $N$ in the category $\mmod A$. We have
$$\Ext^r_A(M,N) = \bigopls_{i \in \Z} \Ext_{{\rm gr}A}^r(M,N[i]).$$
In particular, $\Hom_A(M,N)$ coincides with $\bigopls_{i \in \Z} \Hom_{{\rm gr}A}(M,N[i])$.
It is clear that for $M \in {\rm gr}A$, the projective dimension of $M$ in ${\rm gr}A$ coincides with ${\rm pd}_A M$. All these facts mean that the homological algebra in the category gr$A$ looks very similar to the homological algebra in $\mmod(B)$ for a Noetherian semiperfect algebra $B$.

\medskip

We fix $e$ an idempotent of degree $0$ of $A$. As before, we set $\Gamma:= (1-e)A(1-e)$ and we set $S_e = eA/e{\rm rad}A$, which is the semi-simple graded right $A$-module of degree $0$ supported at $e$. The algebra $A$ being positively graded implies that $\Gamma = (1-e)A(1-e)$ is positively graded as well. The grading on $\Gamma$ is given by
$$\Gamma = (1-e)A_0(1-e) \oplus (1-e)A_1(1-e) \oplus \cdots$$
By \cite[Prop. 2.3]{Sando}, $\Gamma$ is again Noetherian and we may assume that the $e_iA$ are pairwise non-isomorphic. Observe that this induced grading does not necessarily implies that $\Gamma$ is generated in degree $0,1$. Indeed, $\Gamma$ is generated in degree $0,1,2$ when $\Ext^1(S_e,S_e)=0$.

Now, the functor $F=\Hom_A((1-e)A, -): \mmod A \to \mmod \Gamma$ induces a functor
$F_{\rm gr} = \Hom_{A}((1-e)A, -): {\rm gr}A \to {\rm gr}\Gamma$
at the level of the graded categories of right-modules. The left adjoint $G = - \otimes_\Gamma (1-e)A$ to $F$ also induces a functor
$G_{\rm gr} = - \otimes_\Gamma (1-e)A: {\rm gr}\Gamma \to {\rm gr}A$
such that if  $M_\Gamma =\bigopls_{i \in \Z} M_i$ is a graded finitely generated right $\Gamma$-module, then $M \otimes_\Gamma (1-e)A$ is a graded finitely generated right $A$-module such that $(M \otimes_\Gamma (1-e)A)_t$ is the $k$-vector space generated by the elements $$\{m_i \otimes_\Gamma (1-e)a_j \mid m_i \in M_i, a_j \in A_j, i+j = t\}.$$ In the proof of Proposition \ref{HomologicalProp}, the key fact was that $F,G$ induce quasi-inverse equivalences between add$((1-e)A)$ and add$(\Gamma_\Gamma)$.  The same is true for $F_{\rm gr}, G_{\rm gr}$ if we consider instead the graded additive categories add$_{\rm gr}((1-e)A)$ and add$_{\rm gr}(\Gamma_\Gamma)$. For simplicity, when there is no risk of confusion, $F_{\rm gr}$ and $G_{\rm gr}$ will simply be denoted by $F$ and $G$, respectively. Hence, Proposition \ref{HomologicalProp} easily extends to the graded case as follows.

\begin{Prop} \label{HomologicalPropGraded}
Let $M, P, S \in {\rm gr} A$ with $P$ indecomposable projective and $S$ simple.
\begin{enumerate}[$(1)$]
    \item If $P$ is not isomorphic to a direct summand of a shift of $eA$, then $F(P)$ is a graded indecomposable projective module.
\item If $S$ is not isomorphic to a direct summand of a shift of $S_e$, then $F(S)$ is a graded simple module.
\item The functor $F$ is essentially surjective.
\item If $\Ext^i_A(M,S_e) = 0$ for all $i \ge 0$ and $\cdots \to P_1 \to P_0 \to M \to 0$ is a minimal graded projective resolution of $M$, then $\cdots \to F(P_1) \to F(P_0) \to F(M) \to 0$ is a minimal graded projective resolution of $F(M)$.
\item $F(S_e)=0$.
\end{enumerate}
\end{Prop}

For finding the global dimension of a positively graded algebra, by \cite{Roy}, we only need to look at the projective dimensions of graded simple modules,
that is,
$${\rm gl.dim}A = {\rm max} \left\{{\rm pd}\left(\frac{e_iA}{e_i{\rm rad}A}\right) \Big| i = 1,2, \ldots, n\right\}.$$

Before going further, we need to study filtered modules. For $i \ge 0$, set $F_i = A_0 \oplus \cdots \oplus A_i$, so that we get a chain
$$F_0 \subseteq F_1 \subseteq F_2 \subseteq \cdots$$
of $k$-vector spaces whose union is $A$. We have $F_iF_j \subseteq F_{i+j}$ for all $i,j \ge 0$. Hence, $A$ is a \emph{filtered algebra}. A \emph{filtrated right $A$-module} $M$ is just a right $A$-module $M$ with an ascending chain $$M_0 \subseteq M_1 \subseteq M_2 \subseteq \cdots$$ of $k$-vector spaces whose union is $M$ and such that $M_iF_j \subseteq M_{i+j}$ for all $i,j\ge 0$.
Given any finitely generated (but not necessarily graded) right $A$-module $N$, let $g_1, \ldots, g_t$ be a fixed finite set of generators of $N$. For $i \ge 0$, set $N_i = \sum_j g_j \cdot F_i$, so that we have a chain
$$N_0 \subseteq N_1 \subseteq \cdots$$
of $k$-vectors spaces whose union is $N$. Clearly, $N_iF_j \subseteq N_{i+j}$, so that $N$ becomes a filtered $A$-module. Thus, every finitely generated right $A$-module admits a structure of a filtered module. Given $N \in \mmod A$, there are many choices of filtrations on $N$ that give a filtered module structure. The filtration given above, which depends on the chosen set of generators, is called a \emph{standard filtration} of $N$. Given two filtered $A$-modules $M=\cup_{i \ge 0} M_i$ and $N = \cup_{i \ge 0}N_i$, a \emph{filtered morphism} $f: M \to N$ is a morphism of $A$-modules that satisfies $f(M_i) \subseteq N_i$ for all $i \ge 0$. It is called \emph{strict} if for all $i$, $f(M_i) = f(M)\cap N_i$. We get a category fil$(A)$ whose objects are all the finitely generated filtered right $A$-modules and morphisms are the filtered morphisms.  Let Gr$A$ denote the category of all graded $A$-modules. There is a functor
$${\rm gr}: {\rm fil}(A) \to {\rm Gr}A$$
such that for $M = \cup_{i \ge 0}M_i$ a finitely generated filtered $A$-module, gr$M$ is defined so that $({\rm gr}M)_i = M_i/M_{i-1}$ for $i\ge 1$ and $({\rm gr}M)_0 = M_0$. The structure of gr$M$ as a graded $A$-module is the obvious one. Observe that even if $M=\cup_{i\ge 0}M_i$ is finitely generated, gr$M$ may not be finitely generated. However, if $\{M_i\}_{i \ge 0}$ is a standard filtration of $M$, then gr$M$ is finitely generated. The following, which is a particular case of a result due to Roy (see \cite{Roy}), can be found in \cite[page 258]{RobsonMcConnell}.

\begin{Prop} \label{PropFiltered}Let $M$ be a finitely generated $A$-module, and choose a standard filtration of it in order to build ${\rm gr}M \in {\rm gr}A$. Let
$$(*): \quad Q'_t \stackrel{d_t}{\to} Q'_{t-1} \stackrel{d_{t-1}}{\to} \cdots \stackrel{d_1}{\to} Q'_0 \stackrel{d_0}{\to} {\rm gr}M \to 0$$
be a finitely generated graded free resolution of ${\rm gr}M$.
\begin{enumerate}[$(1)$]
    \item There is an exact sequence
$$(**): \quad Q_t \stackrel{f_t}{\to} Q_{t-1} \stackrel{f_{t-1}}{\to} \cdots \stackrel{f_1}{\to} Q_0 \stackrel{f_0}{\to} M \to 0$$
in ${\rm fil}(A)$ where the $Q_i$ are finitely generated free filtered and all the maps are strict, so that $(**)$ is sent to $(*)$ by ${\rm gr}$.
    \item If ${\rm ker}(d_t)$ is projective in ${\rm gr}A$, then ${\rm ker}(f_t)$ is projective.
\end{enumerate}
\end{Prop}

For a graded $A$-module $M$, one defines $d_e(M)$ in a similar way: $d_e(M)$ is the maximal integer $i$ for which $\Ext^i_A(M,S_e)$ is non-zero (if $\Ext_A^i(M,S_e) =0$ for all $i \ge 0$, we set $d_e(M)=-1$; if $\Ext_A^i(M,S_e) \ne 0$ for infinitely many $i$, we set $d_e(M) = \infty$). We first have to make sure that the supremum of the $d_e(M)$ where $M \in {\rm gr}A$ gives the injective dimension of $S_e$.

\begin{Lemma}
We have ${\rm sup}\{d_e(M) \mid M \in {\rm gr}A\} = {\rm id}_A S_e$.
\end{Lemma}

\begin{proof}
We may restrict to the case where $e$ is primitive. First, we have ${\rm sup}\{d_e(M) \mid M \in {\rm gr}A\} \le {\rm id}_A S_e$. If the inequality is strict, then there exists a finitely generated $A$-module $M$, which is not graded, such that $\Ext^r_A(M,S_e) \ne 0$, where $r > {\rm sup}\{d_e(M) \mid M \in {\rm gr}A\}$. Now, consider a standard filtration for $M$ with a finitely generated graded free resolution
$$(*): \quad Q'_{r+1} \stackrel{d_{r+1}}{\to} Q'_r \stackrel{d_r}{\to} \cdots \to Q'_0 \to {\rm gr}M \to 0$$
of ${\rm gr}M$.
Since ${\rm gr}M$ is finitely generated graded, $\Ext^r_A({\rm gr}M, S_e)=0$.
By Proposition \ref{PropFiltered}, there is an exact sequence
$$(**): \quad Q_{r+1} \stackrel{f_{r+1}}{\to} Q_r \stackrel{f_r}{\to} Q_{r-1} \stackrel{f_{r-1}}{\to} \cdots \stackrel{f_1}{\to} Q_0 \stackrel{f_0}{\to} M \to 0$$
in ${\rm fil}(A)$ where the $Q_i$ are finitely generated free filtered and all the maps are strict, so that $(**)$ is sent to $(*)$ by ${\rm gr}$. Since $\Ext^r_A(M,S_e) \ne 0$, there exists a morphism $f: Q_r \to S_e$ such that $ff_{r+1} = 0$ and $f$ does not factor through $f_r$. Since $S_e$ is one dimensional, the morphism $f: Q_r \to S_e$ gives rise to a strict filtered epimorphism, also denoted by $f$.
By applying the functor gr, we get $d_{r+1}{\rm gr}(f)=0$ where ${\rm gr}(f): Q'_r \to S_e[t]$ for some $t$. Since $\Ext^r_A({\rm gr}M, S_e)=0$, there exists $g': Q'_{r-1} \to S_e[t]$ such that $g'd_r = {\rm gr}(f)$. Now from \cite[Chap. 7, Prop. 6.15]{RobsonMcConnell}, there exists a strict filtered morphism $g: Q_{r-1} \to S_e$ such that gr$(g) = g'$. Hence, we get gr$(f-gf_r)={\rm gr}(f)-g'd_{r}=0$. If $f-gf_r$ is non-zero, then it is a strict epimorphism from $Q_r$ to $S_e$, and hence, by \cite[Chap. 7, Cor. 6.14]{RobsonMcConnell}, gr$(f-gf_r)$ is an epimorphism, a contradiction. This shows that $f-gf_r=0$, a contradiction.
\end{proof}

The above lemma actually tells us that the injective dimension of $S_e$ in gr$A$ (or Gr$A$, the category of not necessarily finitely generated graded modules) coincides with the injective dimension of $S_e$ in $\mmod(A)$ (or Mod$(A)$). Note that the corresponding result is not true, in general, for an arbitrary object $M$ in gr$A$.

\section{Homological dimensions} \label{sectionHomDim}

In this section, the algebras considered are Noetherian and they are either semiperfect or positively graded. Sometimes, we restrict to the Artinian case, that is, the finite dimensional case. As a first step to our investigation, we want to relate the homological dimensions of $A$-modules with those of the $\Gamma$-modules. More precisely, we want to relate the global dimensions of $A$ and $\Gamma$ with the homological properties of the semi-simple module $S_e$. Later in this section, we will assume that $e$ is primitive, but at this stage, $e$ is any idempotent of $A$, which is of degree $0$ if $A$ is positively graded.
To unify the notations, we denote by $\C(A)$ and $\C(\Gamma)$ the following categories. If $A$ is semiperfect, $\C(A) = \mmod A$ and $\C(\Gamma) = \mmod \Gamma$ and if $A$ is positively graded, $\C(A) = {\rm gr}A$ and $\C(\Gamma) = {\rm gr}\Gamma$. Recall that we have a functor $F: \C(A) \to \C(\Gamma)$ having nice properties, see Propositions \ref{HomologicalProp} and \ref{HomologicalPropGraded}. We start by relating the projective dimension of an object $M$ in $\C(A)$ to that of $F(M)$ in $\C(\Gamma)$.
We say that $M \in \C(A)$ is \emph{self-orthogonal} if we have $\Ext^i_A(M,M)=0$ for all positive integers $i$.
The following lemma is an analogue to the well known Horseshoe Lemma and will be very handy in the sequel.

\begin{Lemma} \label{ConstructionProjResol}
Let $0 \to L \to M \to N \to 0$ be a short exact sequence in $\C(A)$. Let $$\cdots \to P_1 \to P_0 \to L \to 0$$ and  $$\cdots \to Q_1 \to Q_0 \to M \to 0$$ be projective resolutions of $L$ and $M$ in $\C(A)$, respectively.  Then there exists a projective resolution
$$ \cdots \to P_1 \oplus Q_2 \to P_0 \oplus Q_1 \to Q_0\to N \to 0$$
of $N$ in $\C(A)$.
\end{Lemma}

\begin{proof}
Observe that in the bounded derived category of ${\rm mod} A$, one can replace $L$ and $M$ by their respective projective resolutions, which are in $\C(A)$.  The short exact sequence given gives rise to a triangle
$$L \to M \to N \to L[1].$$
Hence $N$ is quasi-isomorphic to the cone of the morphism $L \to M$. Moreover, this cone is a complex of graded modules if $A$ is positively graded. The result is clear from this: the cone obtained has zero cohomology in all degrees but zero (where it is isomorphic to $N$), since it is quasi-isomorphic to $N$. Hence, the cone is a projective resolution of $N$.
\end{proof}

The following is essential.

\begin{Prop} \label{firstprop}
Let $M$ be in $\C(A)$. Then
$${\rm pd}_{\Gamma}F(M) \le {\rm max}(d_e(M)+{\rm pd}_\Gamma F(eA),{\rm pd}_A M).$$ Moreover, ${\rm pd}_{\Gamma}F(M) = {\rm pd}_A M$ whenever $d_e(M)+{\rm pd}_\Gamma F(eA) < {\rm pd}_A M-1$ or $d_e(M)=-1$.
\end{Prop}

\begin{proof}
If $d_e(M)=-1$, then $\Ext^i_A(M,S_e)=0$ for $i \ge 0$ and we have already observed that ${\rm pd}_{\Gamma}F(M) = {\rm pd}_A M$. We may assume that pd$_A M = r < \infty$ and pd$_{\Gamma}F(eA) = s < \infty$. Then $0 \le d_e(M) \le r$. Let $$\mathcal{P}_M: \quad 0 \to P_r\to \cdots \to P_1 \to P_0\to M \to 0$$ be a minimal projective resolution of $M$ in $\C(A)$, and suppose that for $0 \le i \le r$, we have that $P_i \cong T_i\oplus Q_i$ where $Q_i$ has no direct summand isomorphic to (a shift of) a direct summand of $eA$. Then $d_e(M)$ is the maximal $i \ge 0$ such that $T_i \ne 0$. Note that $F(Q_i)$ is a projective object in $\C(\Gamma)$ for all $i$. Let
$$\mathcal{P}_{F(eA)}: \quad 0 \to R_s\to \cdots \to R_1 \to R_0 \to F(eA) \to 0$$ be a minimal projective resolution of $F(eA)$ in $\C(\Gamma)$.
By using Lemma \ref{ConstructionProjResol} many times, we see that we can get a projective resolution of $F(M)$ in $\C(\Gamma)$ of length at most max$(d_e(M)+s,r)$ using the $F(Q_i)$ and summands of the $R_j$ as terms.

Suppose now that $d_e(M) + s < r-1$. Let $N$ be the $(d_e(M)+1)$-syzygy of $M$ in $\C(A)$.  Then pd$_A N = r-d_e(M)-1 > s$ and the minimal projective resolution of $N$ in $\C(A)$ does not contain (a shift of) a direct summand of $eA$ as a summand.  Therefore, applying $F$ to it, we get a minimal projective resolution of $F(N)$ in $\C(\Gamma)$ of length $r-d_e(M)-1$ and hence, pd$_\Gamma F(N) = r-d_e(M)-1$. Consider now the short exact sequence
$$0 \to F(N) \to F(Q_{d_e(M)}) \oplus F(T_{d_e(M)}) \to F(L) \to 0,$$
in $\C(\Gamma)$ where $L$ is the $d_e(M)$-syzygy of $M$. By Lemma \ref{ConstructionProjResol}, we get a projective resolution of $F(L)$ in $\C(\Gamma)$ of length $r-d_e(M)$ where the last map is the last map of the minimal projective resolution of $F(N)$, since pd$_\Gamma F(N) > {\rm pd}_\Gamma F(Q_{d_e(M)}) \oplus F(T_{d_e(M)})$. Hence, this projective resolution of $F(L)$ is of minimal length and pd$_\Gamma F(L) = r-d_e(M) > s$. By induction, we get pd$_\Gamma F(M)=r$.
\end{proof}

As already observed, the supremum of $\{d_e(M) \mid M \in \C(A)\}$ is the injective dimension of $S_e$ in $\mmod A$.  Therefore, we have the following as a consequence, compare with \cite[Cor. 8.1 (viii)]{Psaroudakis}.

\begin{Cor} \label{firstcoro}
We have ${\rm gl.dim} \Gamma \le  {\rm max}({\rm id}_A S_e + {\rm pd}_{\Gamma}F(eA),{\rm gl.dim} A)$, and we have ${\rm gl.dim} \Gamma = {\rm gl.dim} A$ if ${\rm id}_A S_e+{\rm pd}_{\Gamma}F(eA) < {\rm gl.dim} A -1$.
\end{Cor}

\begin{proof} Observe first that the functor $F$ is essentially surjective by Propositions \ref{HomologicalProp} and \ref{HomologicalPropGraded}. Moreover, the global dimension of an algebra can be reduced to taking the supremum of the projective dimension of the simple objects in $\C(A)$. Thus, $${\rm sup} \{{\rm pd}_\Gamma F(M) \mid M \in \C(A)\} = {\rm gl.dim} \Gamma.$$ Therefore, by Proposition \ref{firstprop}, we have $${\rm gl.dim} \Gamma \le  {\rm max}({\rm sup}\{d_e(M) \mid M \in \C(A)\} + {\rm pd}_{\Gamma}F(eA),{\rm sup}\{{\rm pd}_A(M) \mid M \in \C(A)\}),$$
and this yields the first part of the statement. For the second part, suppose ${\rm id}_A S_e+{\rm pd}_{\Gamma}F(eA) < {\rm gl.dim} A -1$. Let $M$ be an object in $\C(A)$ such that pd$_A M = {\rm gl.dim} A$. We have $$d_e(M)+{\rm pd}_\Gamma F(eA) \le {\rm id}_A S_e +{\rm pd}_\Gamma F(eA) < {\rm gl.dim} A -1 = {\rm pd}_A M-1.$$ Therefore, from Proposition \ref{firstprop}, pd$_A M = {\rm pd}_\Gamma F(M)$. This gives pd$_\Gamma F(M) = {\rm gl.dim} A$. Thus, gl.dim$\Gamma= {\rm gl.dim} A$.
\end{proof}

Note that the bound obtained contains a number that depends on $\Gamma$, namely ${\rm pd}_\Gamma F(eA)$. In general, we cannot replace the latter by a number that does not depend on $\Gamma$. Indeed, in general, ${\rm gl.dim} A < \infty$ does not imply ${\rm gl.dim} \Gamma < \infty$. However, if $S_e$ is self-orthogonal, then the first syzygy $\Omega$ of $S_e$ satisfies $\Ext^i_A(\Omega, S_e) = 0$ for all $i \ge 0$. As observed above, this gives ${\rm pd}_{\Gamma}F(eA) = {\rm pd}_{\Gamma}F(\Omega) = {\rm pd}_{A}\Omega$.
Hence, we get the following.

\begin{Prop} \label{prop1}
Suppose that $S_e$ is self-orthogonal.
\begin{enumerate}[$(1)$]
    \item If $S_e$ is not projective, then ${\rm pd}_{\Gamma}F(eA) = {\rm pd}_{A}S_e - 1$.
\item We have ${\rm gl.dim} \Gamma \le  {\rm max}({\rm id}_A S_e + {\rm pd}_A S_e -1,{\rm gl.dim} A)$, with equality if ${\rm id}_A S_e + {\rm pd}_A S_e < {\rm gl.dim} A$.
\item If ${\rm gl.dim} A < \infty$, then ${\rm gl.dim} \Gamma < \infty$.\end{enumerate}
\end{Prop}

Part (3) of Proposition \ref{prop1} can also be deduced from results in \cite{APT} or in \cite{FullerSaorin}, in the finite dimensional case.

\begin{Remark} In the terminology of \cite{APT}, when $A$ is finite dimensional, then $S_e$ is self-orthogonal if and only if the idempotent ideal $A(1-e)A$ is strong idempotent: every indecomposable summand in the minimal projective resolution of the right $A$-module $A(1-e)A$ is in add$((1-e)A)$. However, the above bound seems not to appear in that paper.
\end{Remark}

For $M \in \mmod A$, $\Omega(M)$ denotes the first syzygy of $M$. If $M \in \C(A)$, then $\Omega(M) \in \C(A)$. For finding a bound on gl.dim$A$, we have the following.

\begin{Prop} \label{Prop2}
Let $M \in \C(A)$, and set $r = {\rm id}_A S_e$.
\begin{enumerate}[$(1)$]
    \item If $r$ is finite, then ${\rm pd}_A M \le r + {\rm pd}_\Gamma F(\Omega^{r+1}(M)) + 1$.
\item We have ${\rm gl.dim} A \le r + {\rm gl.dim} \Gamma + 1$.\end{enumerate}
\end{Prop}

\begin{proof}
Statement $(2)$ follows from statement $(1)$. Assume that $r < \infty$. Let $L = \Omega^{r+1}(M)$, that is, $L$ is the $(r+1)$-th syzygy of $M$ in $\C(A)$.  Then $\Ext^i_A(L,S_e)=0$ for all $i$.  Applying $F$ to a minimal projective resolution of $L$ yields a minimal projective resolution of $F(L)$. Hence, ${\rm pd}_A L = {\rm pd}_\Gamma F(L)$. Thus, ${\rm pd}_A M \le r+1 + {\rm pd}_A L = r + {\rm pd}_\Gamma F(L) + 1$.
\end{proof}

Since the left global dimension of $A$ coincides with the right global dimension of $A$, in the above proposition, we can replace the injective dimension of $S_e = eA/e{\rm rad}A$ by the injective dimension of $Ae/{\rm rad}Ae$. In case $A$ is finite dimensional, id$_A (Ae/{\rm rad}Ae) = {\rm pd}_A S_e$. Thus, the following result follows immediately.

\begin{Prop} \label{LastProp}
Assume that $A$ is finite dimensional. Then $${\rm gl.dim} A \le {\rm min}({\rm id}_A S_e, {\rm pd}_A S_e) + {\rm gl.dim} \Gamma + 1.$$
\end{Prop}

The following lemma is essential for describing the kernel of $F$. It will be particularly useful when $e$ is a primitive idempotent. For $M \in \C(A)$, we denote by add$(M)$ the modules which are direct summands of finite direct sums of copies of (shift of) $M$ in $\C(A)$.

\begin{Lemma} \label{SNLC}
Suppose that $\Ext^1_A(S_e,S_e) = 0$. Then, for $M \in \C(A)$, $F(M)=0$ if and only if $M \in {\rm add}(S_e)$.
\end{Lemma}

\begin{proof}
Suppose that $M \in\C(A)$ is indecomposable and $F(M)=0$.  Let $I = A(1-e)A$.  Then $M \cong M/MI$, that is, $M$ is a $A/I$-module. Observe that $M$ has as a projective cover $P \to M$ in $\C(A)$ with $P$ a finite direct sum of (shift of) copies of $eA$. Hence, we have an epimorphism $P/PI \to M$. Since $\Ext^1_A(S_e,S_e) = 0$, we have that $P/PI$ is in add$(S_e)$, and so is $M$.
\end{proof}

In the above lemma, when $A$ is finite dimensional and $e$ is primitive, ${\rm pd}_A S_e < \infty$ is enough to guarantee the condition $\Ext^1_A(S_e,S_e) = 0$; see \cite{ILP}. However, when $A$ is not finite dimensional, the condition ${\rm pd}_A S_e < \infty$ is not sufficient.
Observe also that when $A$ is not finite dimensional, in Proposition \ref{Prop2}, we cannot replace ${\rm id}_A S_e$ by ${\rm pd}_A S_e$ using a duality argument: there may not be a duality between $\mmod A$ and $\mmod A^{\,\rm op}$. However, we still get the following.

\begin{Prop} \label{LastProp2}
Assume $\Ext^1_A(S_e, S_e)=0$ and let $M \in \C(A)$. Then
\begin{enumerate}[$(1)$]
    \item ${\rm pd}_A M \le {\rm pd}_A S_e + {\rm pd}_\Gamma F(M) + 1.$
\item ${\rm gl.dim} A \le {\rm pd}_A S_e + {\rm gl.dim} \Gamma + 1.$
\item ${\rm gl.dim} A \le {\rm min}({\rm id}_A S_e, {\rm pd}_A S_e) + {\rm gl.dim} \Gamma + 1.$
\end{enumerate}
\end{Prop}

\begin{proof}
Statement (3) follows from Statement (2) and Proposition \ref{Prop2}. Statement (2) follows from Statement (1). To prove Statement (1), we may assume that ${\rm pd}_\Gamma F(M) = m  < \infty$ and ${\rm pd}_A S_e = r < \infty$. Since $\Ext^1_A(S_e, S_e)=0$, we see that there exists a short exact sequence
$$\eta: \;\; 0 \to N \to M \to S \to 0$$
in $\C(A)$ where $S\in {\rm add}(S_e)$ such that the top of $N$ does not contain (a shift of) a direct summand of $S_e$ as a direct summand. Let $P_0 \to N \to 0$ be a projective cover in $\C(A)$ with kernel $K$. Then $F(P_0) \to F(N) \to 0$ is also a projective cover in $\C(\Gamma)$. We proceed by induction on $m$. Suppose first that $m=0$. Then we have $F(K)=0$. From Lemma \ref{SNLC}, we get that $K \in {\rm add}(S_e)$. Thus, ${\rm pd}_A K \le r$, and this gives ${\rm pd}_A N \le r+1$. The short exact sequence $\eta$ gives ${\rm pd}_A M \le r+1$ as wanted. Now, assume $m > 0$. We have ${\rm pd}_\Gamma F(K) \le m-1$. Therefore, by induction, ${\rm pd}_A K \le r+m$ and using the same argument as above, we get ${\rm pd}_A M \le r + m + 1$.
\end{proof}

\begin{Cor}
Suppose that $S_e$ is self-orthogonal. Then $${\rm gl.dim} A \le 2{\rm gl.dim} \Gamma + 2$$ and  $${\rm gl.dim} \Gamma \le  {\rm max}({\rm id}_A S_e + {\rm pd}_A S_e -1,{\rm gl.dim} A).$$
\end{Cor}

\begin{proof}
The second inequality is just Proposition \ref{prop1}. Let $\Omega$ be the first syzygy of $S_e$.  Then $\Ext^i_A(\Omega, S_e)=0$ for all non-negative integers $i$. This means that ${\rm pd}_A \Omega = {\rm pd}_\Gamma F(\Omega) \le {\rm gl.dim} \Gamma$. Now from Proposition \ref{LastProp2},
\begin{eqnarray*}
{\rm gl.dim} A &\le&  {\rm pd}_A S_e + {\rm gl.dim} \Gamma + 1\\
& = & {\rm pd}_A \Omega + {\rm gl.dim} \Gamma + 2\\
& \le & 2{\rm gl.dim} \Gamma + 2.
\end{eqnarray*} \qedhere
\end{proof}

The following theorem, in the case of a finite dimensional $k$-algebra, follows from the fact \cite{Howard} that when $S_e$ is self-orthogonal (that is, $\Ext_A^i(S_e, S_e)=0$ for all $i > 0$), then the singularity categories of $A$ and $\Gamma$ are triangle-equivalent.

\begin{Theo}
Let $A$ be Noetherian which is either semiperfect or positively graded. Suppose that $S_e$ is self-orthogonal.  Then ${\rm gl.dim} \Gamma < \infty$ if and only if ${\rm gl.dim} A < \infty$.
\end{Theo}

One question remains: if both ${\rm gl.dim} \Gamma, {\rm gl.dim} A$ are finite, does this imply that $S_e$ is self-orthogonal? As already observed, the answer is no in general, even if $e$ is primitive. We may have $e$ primitive and $\Ext_A^1(S_e,S_e)\ne 0$ with both $A, \Gamma$ of finite global dimensions. This cannot happens when $A$ is finite dimensional, since ${\rm pd}_A S_e < \infty$ implies $\Ext_A^1(S_e,S_e)=0$. So, in the finite dimensional case, we have the following conjecture.

\begin{Conj} \label{conj3}
Let $A$ be a finite dimensional $k$-algebra with $e$ primitive. If both ${\rm pd}_A S_e$ and ${\rm pd}_\Gamma (eA(1-e))$ are finite, then $S_e$ is self-orthogonal.
\end{Conj}

When $A$ is not finite dimensional, we simply have to add the vanishing condition $\Ext_A^1(S_e,S_e)=0$, and we get the following conjecture.

\begin{Conj} \label{conj3infinite}
Let $A$ be a Noetherian $k$-algebra which is either semiperfect or positively graded. Assume that $e$ is primitive (and of degree zero if $A$ is positively graded) with $\Ext_A^1(S_e,S_e)=0$. If both ${\rm pd}_A S_e$ and ${\rm pd}_\Gamma (eA(1-e))$ are finite, then $S_e$ is self-orthogonal.
\end{Conj}

\begin{Prop}[Assuming Conj. \ref{conj3infinite}] \label{conj1infinite}
Let $A$ be a Noetherian $k$-algebra which is either semiperfect or positively graded. Let $e$ be primitive (and of degree zero if $A$ is positively graded) with $\Ext_A^1(S_e,S_e)=0$. If the global dimension of $A$ and $\Gamma$ are finite, then $S_e$ is self-orthogonal.
\end{Prop}

\begin{Prop}[Assuming Conj. \ref{conj3infinite}] \label{conj2infinite}
Let $A$ be a Noetherian $k$-algebra which is either semiperfect or positively graded.  Let $e$ be primitive (and of degree zero if $A$ is positively graded). Any two of the following imply the third.
\begin{enumerate}[$(1)$]
    \item ${\rm gl.dim} \Gamma < \infty$,
    \item ${\rm gl.dim} A < \infty$ and $\Ext_A^1(S_e,S_e)=0$,
    \item $S_e$ is self-orthogonal.
\end{enumerate}
\end{Prop}

The rest of the paper is devoted to proving Conjecture \ref{conj3} and, with the additional assumption that $A/{\rm rad}A$ is finite dimensional, Conjecture \ref{conj3infinite}. Note that this additional assumption holds when $A$ is positively graded, but does not necessarily hold when $A$ is semiperfect.

\section{Main tools} \label{sectionSmallHomDim}

In this section, all algebras considered are Noetherian $k$-algebras and are either semiperfect or positively graded. We provide useful tools for proving Conjecture \ref{conj3infinite}. In this section, $e$ is always assumed to be primitive, and is of degree zero if $A$ is positively graded.
The following two lemmas will be our main tools in the rest of this paper.

\begin{Lemma} \label{cruciallemma}
Assume $\Ext_A^1(S_e,S_e)=0$, ${\rm pd}_A S_e < \infty$ and ${\rm pd}_\Gamma F(eA) < \infty$. Then ${\rm pd}_\Gamma F(eA) \ge {\rm pd}_A S_e-1$ and $d_e(S_e) \le {\rm max}(0,{\rm pd}_A S_e - 2)$.
\end{Lemma}

\begin{proof}
We start with an easy observation. Let $M \in \C(A)$. Then there exists a short exact sequence
$$\eta_M: \;\; 0 \to M' \to M \to S_M \to 0$$
in $\C(A)$ where $S_M \in {\rm add}(S_e)$ is such that $\Hom_A(M', S_e)=0$. Let $P \to M' \to 0$ be a projective cover in $\C(A)$ with kernel $K_M$. Then $F(P) \to F(M') = F(M)$ is a projective cover in $\C(\Gamma)$ and hence, the first syzygy of $F(M)$ is $F(K_M)$. Now, let ${\rm pd}_\Gamma F(eA) = m$ and ${\rm pd}_A S_e = r$, where we may assume $r \ge 2$. If $m=0$, then the second syzygy of $S_e$ lies in add$(S_e)$. Since pd$_A S_e$ is finite, this means that $r \le 1$. Hence, we may assume $m > 0$.  Consider a minimal projective presentation $P \to eA \to S_e \to 0$ in $\C(A)$ of $S_e$ where we know that (a shift of) $eA$ is not a direct summand of $P$. Let the second syzygy of $S_e$ be $M_0$, so pd$_A M_0 = r-2$ and $M_0 \ne 0$. If $S_{M_0} \ne 0$, then  pd$_A M_0' = r-1$ and otherwise, $M_0' = M_0$ and pd$_A M_0' = r-2$. Let $M_1 = K_{M_0}$, so pd$_A M_1 = r-2$ if $S_{M_0} \ne 0$ and pd$_A M_1 = r-3$ if $S_{M_0} = 0$. For $1 \le i \le m$, we let $M_{i} = K_{M_{i-1}}$. We can prove by induction that pd$_A M_i \le r-2$ and pd$_A M_i' \le r-1$ for $0 \le i \le m-1$. At the last step, $F(M_{m-1})$ is projective in $\C(\Gamma)$, hence $M_m=K_{M_{m-1}} \in {\rm add}(S_e)$. But pd$_A K_{M_{m-1}} = {\rm pd}_A M'_{m-1}-1 \le r-2$. Thus, $K_{M_{m-1}} = 0$ and hence, $M'_{m-1}$ is projective in $\C(A)$. Since $r > 0$, this gives $S_{M_{m-1}}=0$ so $M_{m-1} = M'_{m-1}$. We can prove using another (reverse) induction that pd$_A M_i = m-1 - i$ and $M_i = M'_i$ for $i = m-1, m-2, \ldots, m-r+1$. In particular, $m-r+1 \ge 0$. This proves the first part of the proposition. If $m-r+1 = 0$, then $M_0$ is such that $\Ext^i_A(M_0, S_e)=0$ for all $i \ge 0$ and thus, $S_e$ is self-orthogonal and this proves the second part of the proposition in this case. Assume $m-r+1 > 0$. This gives pd$_A M'_{m-r} = r-1$. Now, the exact sequence
$$\eta_{M_{m-r}}: \;\; 0 \to M'_{m-r} \to M_{m-r} \to S_{M_{m-r}} \to 0$$
with $S_{M_{m-r}} \in {\rm add}(S_e)$ gives $S_{M_{m-r}} \ne 0$ since otherwise, pd$_A M_{m-r} = r-1$, contrary to what we have proven so far. Observe that $\Ext^i_A(M'_{m-r}, S_e)=0$ for all $i \ge 0$. Using $\eta_{M_{m-r}}$ together with Lemma \ref{ConstructionProjResol}, we see that we can get a projective resolution
$$0 \to Q_r \to Q_{r-1} \to \cdots \to Q_0 \to S_{M_{m-r}} \to 0$$
in $\C(A)$ such that $Q_r, Q_{r-1}$ are the last two terms in a minimal projective resolution of $M'_{m-r}$. Thus, (a shift of) $eA$ is not a direct summand of $Q_r \oplus Q_{r-1}$. Since the latter resolution is of minimal length, this gives $\Ext^i_A(S_e, S_e)=0$ for $i=r-1, r$.
\end{proof}

\begin{Lemma} \label{technical}
Suppose that ${\rm pd}_A S_e < \infty$, $\Ext^1_A(S_e,S_e)=0$ and ${\rm pd}_\Gamma F(eA) < \infty$. If $S_e$ is not self-orthogonal then $\Ext^{d_e(S_e)-1}_A(S_e,S_e)\ne 0$.
\end{Lemma}

\begin{proof}
Suppose that ${\rm pd}_\Gamma F(eA)=r$.
Let $M_i$ in $\C(A)$ be the $(d_e(S_e)+i)$-syzygy of $S_e$.  We know that $M_1$ is nonzero by Lemma \ref{cruciallemma}. We have a short exact sequence
$$(*): \quad 0 \to F(M_1) \to Q \oplus R \to F(M_0) \to 0,$$
in $\C(\Gamma)$ where $Q$ is a projective object in $\C(\Gamma)$ and $R$ is nonzero in add$(F(eA))$. By assumption, $d_e(S_e) > 0$, and hence $d_e(S_e) \ge 2$. Suppose to the contrary that $\Ext^{d_e(S_e)-1}_A(S_e,S_e)$ vanishes.  Set $t = {\rm pd}_A S_e - d_e(s_e)-1 \ge 1$, which is the projective dimension of $F(M_1)$. By Lemma \ref{cruciallemma}, we know that $r \ge {\rm pd}_A S_e -1 \ge t + 2$, since $d_e(S_e) \ge 2$. Hence, $t \le r-2$.  Let
$$0 \to P_t \to P_{s-1} \to \cdots \to P_1 \to P_0 \to F(M_1) \to 0$$
be a minimal projective resolution of $F(M_1)$ in $\C(\Gamma)$ and
$$0 \to Q_r \to Q_{r-1} \to \cdots \to Q_1 \to Q_0 \to R \to 0$$
be a minimal projective resolution of $R$ in $\C(\Gamma)$. By Lemma \ref{ConstructionProjResol}, we get a projective resolution
$$(**): \cdots \to Q_{t+2} \to Q_{t+1} \oplus P_t \to Q_t \oplus P_{t-1}\to \cdots \to Q_1\oplus P_0 \to Q\oplus Q_0\to F(M_0) \to 0$$
of $F(M_0)$ in $\C(\Gamma)$, where $Q_{t+1}$ is nonzero. If $t < r-2$, then $(**)$ is clearly of minimal length since the last map is the last map $Q_{r} \to Q_{r-1}$ of the minimal projective resolution of $F(eA)$. If $t = r-2$, the last map is $Q_{t+2} \to Q_{t+1} \oplus P_t$ whose image lies in $Q_{t+1}$.  Thus, this map is a radical map and hence, $(**)$ is of minimal length.
%If $t = r-1$, the last map of (**) is $f: Q_{t+1} \oplus P_t \to Q_t \oplus P_{t-1}$. As for the $t=r-2$ case, the restriction of $f$ to $Q_{t+1}$ is a %radical map, hence not a section. Thus, $f$ is not a section, which means that $(**)$ is of minimal length.
Therefore, in all cases, pd$_\Gamma F(M_0)=r$. Since $\Ext^{d_e(S_e)-1}_A(S_e,S_e)=0$, we have a short exact sequence
$$0 \to F(M_0) \to Q' \to F(M_{-1}) \to 0,$$
in $\C(\Gamma)$ where $Q'$ is projective in $\C(\Gamma)$. Hence, pd$_\Gamma F(M_{-1}) = r+1$. Now, since $r+1 > r$, by using an argument similar as in the first part of the proof, we get pd$_\Gamma F(M_{-2}) = r+2$. By induction, ${\rm pd}_\Gamma F(eA) = {\rm pd}_\Gamma F(M_{-d_e(S_e)+1})=r+d_e(S_e)-1$, which gives $d_e(S_e)=1$, a contradiction.
\end{proof}

\begin{Lemma} \label{NotProjective}
Suppose that ${\rm pd}_A S_e < \infty$, $\Ext_A^1(S_e,S_e)=0$ and $S_e$ is not self-orthogonal. If ${\rm pd}_\Gamma F(eA)$ is finite, then ${\rm pd}_\Gamma F(eA) \ge 3$ and ${\rm pd}_A S_e \ge 4$.
%Moreover ${\rm pd}_\Gamma F(eA) = 1$ only if ${\rm pd}_A S_e = 2$.
\end{Lemma}

\begin{proof}
Assume that ${\rm pd}_\Gamma F(eA) < \infty$. Since $\Ext_A^1(S_e,S_e) = 0$ and $S_e$ is not self-orthogonal, we get pd$_A S_e, d_e(S_e) \ge 2$. From Lemma \ref{cruciallemma}, we have pd$_\Gamma F(eA) \ge {\rm pd}_A S_e -1$ and $d_e(S_e) \le {\rm max}(0,{\rm pd}_A S_e - 2)$. The second inequality gives ${\rm pd}_A S_e \ge 4$. Hence, it follows from the first inequality that ${\rm pd}_\Gamma F(eA) \ge 3$.
%If ${\rm pd}_\Gamma F(eA) = 0$, then ${\rm pd}_A S_e \le 1$, a contradiction. If ${\rm pd}_\Gamma F(eA) = 1$, then ${\rm pd}_A S_e \le 2$. From the observation above, ${\rm pd}_A S_e = 2$.
\end{proof}

\section{The conjecture}

In this section, $A$ is a Noetherian $k$-algebra which is either semiperfect with $A/{\rm rad}A$ finite dimensional or positively graded. Therefore, $A/{\rm rad}A$ is a product of copies of $k$ in both cases. The idempotent $e$ is always assumed to be primitive, and of degree zero if $A$ is positively graded. We prove Conjecture \ref{conj3infinite} in this setup.

\medskip

For a morphism $f: M \to N$ between finitely generated modules, we denote by $\bar f$ the induced morphism $\bar f: M/{\rm rad}M \to N/{\rm rad}N$ on the tops of $M,N$.
The following lemma is quite easy to prove in the finite dimensional setting, but it is not so obvious in our setting.

\begin{Lemma} \label{MinResolution1}
Assume that $A$ is semiperfect with $A/{\rm rad}A$ finite dimensional. Let $M \in \mmod A$ admitting a projective resolution
$$ 0 \to P_r \stackrel{d_r}{\longrightarrow} P_{r-1} \stackrel{d_{r-1}}{\longrightarrow} \cdots \stackrel{d_2}{\longrightarrow} P_1 \stackrel{d_1}{\longrightarrow} P_0 \stackrel{d_0}{\longrightarrow} M \to 0$$
which is minimal in $\mmod A$. Let $f: M \to M$ be a morphism in $\mmod A$ and $\{f_i: P_i \to P_i\}_{0 \le i \le r}$
a lifting of $f$, in $\mmod A$, to the above projective resolution of $M$. If $f$ is a radical morphism, then all of the $\bar f_i$ are nilpotent.
\end{Lemma}

\begin{proof}
Let us consider the projective cover $P_0 \stackrel{d_0}{\to} M \to 0$ of $M$ in $\mmod A$ with the lifting $f_0 : P_0 \to P_0$ of $f$ in $\mmod A$. Since $d_0$ is a projective cover, $\bar d_0$ is an isomorphism. By considering the diagram
$$\xymatrix{P_0 \ar[r]^{d_0} \ar[d]^{f_0} & M \ar[r] \ar[d]^f & 0 \\ P_0 \ar[r]^{d_0} & M \ar[r] & 0}$$
modulo the radical, we easily see that if $f$ is an isomorphism, then so is $\bar f$ and hence $\bar f_0$ is an isomorphism. This gives that $f_0$ is an isomorphism. Conversely, if $f_0$ is an isomorphism, then $f$ is surjective. Since $M$ is a Noetherian module, it is well known that surjectivity of $f$ implies injectivity of $f$. Hence, $f_0$ is an isomorphism if and only if $f$ is. By repeating the argument at the level of the kernel of $d_0$, we see that $f$ is an isomorphism if and only if all of the $f_i$ are isomorphisms.

Assume that $f$ is radical. Then $f_0$ is also a radical morphism. We claim that $1-af$ is invertible for all $a \in k$. By the above observation, it is sufficient to prove that $1-af_0$ is invertible for all $a \in k$.
Since the category $\mathcal{P}(A)$ of the projective objects in $\mmod A$ is Krull-Schmidt, it has a well defined radical $\mathcal{J}(A)$. For $P,Q$ indecomposable in $\mathcal{P}(A)$, $g: P \to Q$ lies in $\mathcal{J}(A)$ if and only if $g$ is a radical map. In general, if $g: P \to Q$ is represented by a matrix, then $g$ lies in $\mathcal{J}(A)$ if and only if each entry is in  $\mathcal{J}(A)$. By the properties of the radical of a category, $f_0$ being in $\mathcal{J}(A)$ means that $1 - hf_0$ is invertible for all morphisms $h : P_0 \to P_0$, and in particular, $1 - af_0$ is invertible for all $a \in k$. This proves our first claim. Since for $a \in k$, the morphisms $1-af, 1-af_0$ are invertible, it follows that for all $a \in k$ and all $i$, the morphisms $1-af_i$ are invertible.

Suppose that $P_i = \bigopls_{j=1}^{t_i}Q_{ij}$ where the $Q_{ij}$ are objects in the list $\{e_1A, \ldots, e_nA\}$. Let $[f_i]$ be the matrix of $f_i$ according to this decomposition. Since $k = \bar k$ and $A/{\rm rad}A$ is finite dimensional, each entry of $[f_i]$ can be written as a scalar times the identity (whenever this makes sense) plus a radical map. Hence, $[f_i] = E_i + F_i$, where $F_i$ is a matrix containing only radical maps and $E_i$ is a matrix containing scalar multiples of identities. The matrix of $1-af_i$ is $I-aE_i-aF_i$. This is invertible for all $a \in k$ and hence, the matrix $I-aE_i$ of $\overline{1 - af_i}$ is invertible for all $a \in k$. This implies that $E_i$ is nilpotent. Hence, $\bar f_i$ is nilpotent.
\end{proof}

\begin{Lemma} \label{MinResolution2}
Assume that $A$ is semiperfect with $A/{\rm rad}A$ finite dimensional. Let $M \in \mmod A$ admitting a projective resolution
$$\mathcal{P}: \quad 0 \to P_r \stackrel{d_r}{\longrightarrow} P_{r-1} \stackrel{d_{r-1}}{\longrightarrow} \cdots \stackrel{d_2}{\longrightarrow} P_1 \stackrel{d_1}{\longrightarrow} P_0 \stackrel{d_0}{\longrightarrow} M \to 0$$
which is minimal in $\mmod A$. Let $f: M^s \to M^t$ be a morphism in $\mmod A$ and $\{f_i: P_i^s \to P_i^t\}_{0 \le i \le r}$
a lifting of $f$, in $\mmod A$, to the projective resolutions $\mathcal{P}^s$ and $\mathcal{P}^t$ of $M^s$ and $M^t$, respectively. If $f$ is a radical morphism, then none of the $f_i$ are sections.
\end{Lemma}

\begin{proof} Assume that $f$ is radical. Then all the components $M \to M^t$ of $f$ are radical. Also, if some $f_i$ is a section, then all components $P_i \to P_i^t$ of $f_i$ are sections. Therefore, we may assume that $s=1$.
Write $f = (f^1, f^2, \ldots, f^t)^T: M \to M^t$ and for $0 \le i \le r$, write $f_i = (f_i^1, f_i^2, \ldots, f_i^t)^T: P_i \to P_i^t$. Observe that for a given $1 \le j \le t$, the morphisms $\{f_i^j: P_i \to P_i\}_{0 \le i \le r}$ form
a lifting of $f^j$, in $\mmod A$, to the given projective resolution of $M$. Fix $0 \le i \le r$. It follows from Lemma \ref{MinResolution1} that the morphisms $\bar{f_i^1}, \ldots, \bar{f_i^t}$ are nilpotent. Consider the Lie subalgebra $\mathfrak{g}$ of $\mathfrak{g}\mathfrak{l}(P_i/{\rm rad}P_i)$ generated by the $\bar{f_i^1}, \ldots, \bar{f_i^t}$. Since a sum of compositions of the morphisms $f^1, \ldots, f^t$ is again radical, it follows that any element in $\mathfrak{g}$ is a nilpotent endomorphism. By Engel's Theorem, there is a common null vector $v$ to all elements of $\mathfrak{g}$. Therefore, the morphism $\bar f_i = (\bar{f_i^1}, \ldots, \bar{f_i^t})^T$ is not a section and hence, $f_i$ is not a section.
\end{proof}

The following lemma is an analogue of Lemma \ref{MinResolution2} in the setting of positively graded $k$-algebras.

\begin{Lemma} \label{LemmaLiftingGraded} Assume that $A$ is positively graded. Let $L \in {\rm gr} A$ be generated in a single degree. Let $M = \bigopls_{i=1}^rL[p_i]$ and $N = \bigopls_{j=1}^sL[q_j]$. Consider a minimal projective resolution $\mathcal{P}_L$ of $L$ in ${\rm gr} A$ and use direct sums of shifts of $\mathcal{P}_L$ to build minimal projective resolutions
$$\mathcal{P}_M: \quad \cdots \to P_2 \to P_1 \to P_0 \to M \to 0$$
and
$$\mathcal{P}_N: \quad \cdots \to Q_2 \to Q_1 \to Q_0 \to N \to 0$$
of $M$ and $N$ in ${\rm gr} A$, respectively. Let $f: M \to N$ be a graded morphism, and for $i \ge 1$, let $f_i: P_i \to Q_i$ be graded morphisms that form a lifting of $f$ to the above projective resolutions. If $f$ is a radical morphism, then none of the $f_i$ are sections.
\end{Lemma}

\begin{proof} Suppose that $f$ is a radical morphism. We may assume that $L$ is generated in degree $0$. Let us fix $m \ge 1$. Let
$$\mathcal{P}_L: \quad \cdots \to R_2 \to R_1 \to R_0 \to L \to 0$$
be a minimal projective resolution of $L$ in gr$A$. Decompose $R_m$ so that $R_m = S_1[j_1] \oplus \cdots \oplus S_t[j_t]$ where the $S_i$ are nonzero projective in gr$A$ that are generated in degree $0$ and $j_1 \le j_2 \le \cdots \le j_t$ are non-negative integers. Observe that
$$P_m = \bigopls_{i=1}^rR_m[p_i] = \bigopls_{i=1}^r\bigopls_{l=1}^tS_l[j_l+p_i],$$
and
$$Q_m = \bigopls_{i=1}^sR_m[q_i] = \bigopls_{i=1}^s\bigopls_{l=1}^tS_l[j_l+q_i].$$
Now, $f_m$ is given by an $st \times rt$ matrix $[f_m]$. If $f_m$ is a section, then every column of $[f_m]$ is a section. Let $q$ be the greatest element of the $q_i$ and $p$ be the greatest element of the $p_i$. We may assume that in the $s \times r$ matrix of $f$, there is no zero row. Since $f$ is a radical morphism and $L$ is generated in a single degree, this means that $p > q$. Consider the column $c$ of $[f_m]$ corresponding to the summand $S_1[j_t + p]$ of $P_m$. Since $j_t + p > j_l + q_i$ for all $1 \le l \le t$, $1 \le i \le s$, we see that $c$ only contains radical morphisms, and hence cannot be a section, a contradiction.
\end{proof}

We are now ready to prove our first main theorem of this section.

\begin{Theo} \label{MainTheo}
Let $A$ be a semiperfect Noetherian $k$-algebra with $A/{\rm rad}A$ finite dimensional and assume that $e$ is primitive with $\Ext_A^1(S_e, S_e) = 0$.  If both ${\rm pd}_A S_e$ and ${\rm pd}_\Gamma F(eA)$ are finite, then $S_e$ is self-orthogonal.
\end{Theo}

\begin{proof}
Assume ${\rm pd}_A S_e = r < \infty$ and ${\rm pd}_\Gamma F(eA) = s < \infty$ and suppose to the contrary that $S_e$ is not self-orthogonal. Write $d := d_e(S_e)$. We have $d, r \ge 2$. The case $r \le 3$ can be excluded using Lemma \ref{NotProjective}. If $r = 4$, then from Lemma \ref{cruciallemma}, we get $d=2$. By Lemma \ref{technical}, we get $\Ext_A^1(S_e, S_e) \ne 0$, which is impossible. So consider the case where $r \ge 5$. Again, we have $d > 2$. By Lemma \ref{technical},
both $\Ext_A^{d-1}(S_e,S_e)$, $\Ext_A^d(S_e,S_e)$ are nonzero. Let $L_i$ be the $i$-th syzygy of $S_e$. By definition of $d$, the syzygy $L_{d+1}$ is such that $\Ext^j_A(L_{d+1}, S_e)=0$ for all $j \ge 0$. Therefore, ${\rm pd}_\Gamma F(L_{d+1}) = {\rm pd}_A L_{d+1} = r-d-1 \ge 1$. Since $s \ge r-1$ by Lemma \ref{cruciallemma}, we get $r-d-1 \le s-d < s-2$.
Consider the following part of a minimal projective resolution
$$0 \to L_{d+1} \to (eA)^p \oplus P_d \to  (eA)^q \oplus P_{d-1} \to L_{d-1} \to 0$$
of $S_e$ in mod$A$ where $eA$ is not isomorphic to a direct summand of $P_d \oplus P_{d-1}$ and both $p,q$ are positive.  Applying the functor $F$, we get an exact sequence
$$0 \to F(L_{d+1}) \to F(eA)^p \oplus F(P_d) \stackrel{f}{\to} F(eA)^q \oplus F(P_{d-1}) \to F(L_{d-1}) \to 0$$
where $F(P_d), F(P_{d-1})$ are projective $\Gamma$-modules and $F(L_{d+1})$ has projective dimension $r-d-1$. Consider now a minimal projective resolution
$$0 \to Q_s \to \cdots \to Q_1 \to Q_0 \to F(eA) \to 0$$
of $F(eA)$, where $s \ge r-1 \ge 4$.  We have a commutative diagram
$$\xymatrixcolsep{10pt}\xymatrixrowsep{14pt}\xymatrix{& &0 \ar[d] & 0 \ar[d]&& \\ & & Q_s^p \ar[d]^{h_s} \ar[r]^{f_s}& Q_s^q \ar[d]^{g_s}&&& \\ &&Q_{s-1}^p\ar@{.}[d] \ar[r]^{f_{s-1}}&Q_{s-1}^q\ar@{.}[d] &&\\ & &F(P_d) \oplus Q_0^p \ar[d] \ar[r]^{f_0}&  F(P_{d-1}) \oplus Q_0^q \ar[d] &&\\0 \ar[r]& F(L_{d+1})\ar[r]& F(P_d) \oplus F(eA)^p \ar[r]^f& F(P_{d-1}) \oplus F(eA)^q \ar[r] & F(L_{d-1}) \ar[r]&0\\}$$
where $\{f_i\}_{0 \le i \le s}$ is a lifting of $f$ to the projective resolutions in the diagram above. By Lemma \ref{ConstructionProjResol}, this yields a projective resolution of length $s+1$ of $F(L_{d-1})$. The tail of this projective resolution is
$$0 \rightarrow Q_s^p \stackrel{u}{\rightarrow} Q_s^q \oplus Q_{s-1}^p$$
where $u = (f_s, h_s)^T$.
%and $$v = \left(%
%\begin{array}{cc}
%  g_s & f_{s-1} \\
%  0 & h_{s-1} \\
%0 & 0\\
%\end{array}%
%\right).$$
Assume that $u$ is a section, which means that $f_s$ is a section, since $h_s$ is a radical map. Decompose $f$ as a $2 \times 2$ matrix whose component corresponding to $F(eA)^p \to F(eA)^q$ is denoted $g$.  Similarly decompose $f_0$ as a $2 \times 2$ matrix whose component corresponding to $Q_0^p \to Q_0^q$ is denoted $h$.  We get a commutative diagram
$$\xymatrix{0 \ar[r] & Q_s^p \ar[d]^{f_s} \ar[r] & Q_{s-1}^p \ar[d]^{f_{s-1}}\ar[r] & \cdots \ar[r] & Q_1^p \ar[d]^{f_1}\ar[r] & Q_0^p \ar[d]^h\ar[r] & F(eA)^p \ar[r]\ar[d]^g & 0\\
0 \ar[r] & Q_s^q \ar[r] & Q_{s-1}^q \ar[r] & \cdots \ar[r] & Q_1^q \ar[r] & Q_0^q \ar[r] & F(eA)^q \ar[r] & 0}$$
Observe that $g$ is a radical map,
%meaning that each component $F(eA) \to F(eA)^q$ of $g$ is radical. Hence, none of the components $Q_s \to Q_s^q$ of $f_s$ can be a section by Lemma %\ref{MinResolution2}.
which gives that $f_s$ is not a section, by Lemma \ref{MinResolution2}. This is a contradiction. Therefore, $u$ is not a section, meaning that ${\rm pd}_\Gamma F(L_{d-1})=s+1$. Since for $P$ finitely generated projective in $\mmod A$, we have ${\rm pd}_\Gamma F(P) \le s$, we get ${\rm pd}_\Gamma F(L_{d-1-i})=s+1+i$ for $0 \le i \le d-2$. For $i=d-2$, we get ${\rm pd}_\Gamma F(L_{1}) = {\rm pd}_\Gamma F(eA) =s+d-1 \ne s$, since $d \ne 1$. This is a contradiction.
\end{proof}

Of course, the above theorem also establishes this weaker version.

\begin{Theo}
Let $A$ be a semiperfect Noetherian $k$-algebra with $A/{\rm rad}A$ finite dimensional and assume that $e$ is primitive with $\Ext_A^1(S_e, S_e) = 0$. If both ${\rm gl.dim} A$ and ${\rm gl.dim \Gamma}$ are finite, then $S_e$ is self-orthogonal.
\end{Theo}

In case $A$ is positively graded, we get the following.

\begin{Theo} \label{MainTheo2}
Let $A$ be a positively graded Noetherian $k$-algebra, and assume that $e$ is primitive of degree zero with $\Ext_A^1(S_e, S_e) = 0$.  If ${\rm pd}_A S_e$ and ${\rm pd}_\Gamma F(eA)$ are finite, then $S_e$ is self-orthogonal.
\end{Theo}

\begin{proof} Same as the proof of Theorem \ref{MainTheo}, by working in the categories ${\rm gr}A, {\rm gr}\Gamma$ rather than $\mmod A, \mmod \Gamma$ and using Lemma \ref{LemmaLiftingGraded} rather than Lemma \ref{MinResolution2}. The modules $(eA)^p, (eA)^q$ in the proof of Theorem \ref{MainTheo} have to be replaced by non-trivial direct sums of shifts of $eA$ and $P_d \oplus P_{d-1}$ do not have a shift of $eA$ as a direct summand.
\end{proof}

Of course, we also have this weaker version.

\begin{Theo}
Let $A$ be a positively graded Noetherian $k$-algebra, and assume that $e$ is primitive of degree zero with $\Ext_A^1(S_e, S_e) = 0$. If both ${\rm gl.dim} A$ and ${\rm gl.dim \Gamma}$ are finite, then $S_e$ is self-orthogonal.
\end{Theo}

\section{Examples}
\subsection{Skew group rings with cyclic groups}
Let $k$ be an algebraically closed field of characteristic zero {and $m \ge 2$ a positive integer.
Let $S=k[x_1,\ldots,x_n]$ and $\mu_m = \{z \in k \st z^m=1 \}.$
We write a  diagonal action of $\mu_m$ on $S$, using superscript notation,
$x_i^\zeta = \zeta^{a_i}x_i$ for $\zeta \in \mu_m$ by choosing $a_1,\ldots,a_n \in \Z/m$ .  Consider the skew group
algebra
$$A=S \rtimes \mu_m = \bigopls_{\zeta \in \mu_m} S\zeta$$
with multiplication $\zeta x_i= x_i^\zeta\zeta.$
Let $\chi_i: \mu_m \rightarrow \mu_m$ be the character $\chi_i(\zeta)=\zeta^i$.
For $i \in \Z/m$, let
$e_i = \frac{1}{m} \sum_{\zeta \in \mu_m} \chi_i(\zeta) \zeta$
be the primitive idempotents of $A$.  Note that $\chi_i(\zeta)$ is a coefficient in $k$ and $\zeta$ is an element of the group $\mu_m$
in this expression of $e_i$ in the group algebra $k\mu_m$ which is a subalgebra of $A$.

It is well known that $A$ can be presented as the path algebra with relations via the McKay graph and commuting relations~\cite{CMT,BSW}. More specifically, we define
$Q_0 =\Z/m$ and we have $n$ arrows that go into and out of each vertex.  The arrows that start from $i$ go
to vertices $i+a_1,\ldots, i+a_n$.  Write $x_i^j$ for the arrow from vertex $j$ to vertex $j+a_i$.  The commuting relations are of the form
$$ x_i^{j+a_k} x^j_k = x_k^{j+a_i} x^j_i.$$
The explicit isomorphism of $A$ with the path algebra $kQ$ modulo these relations can be seen in Proposition 2.8(3) of \cite{CMT} or in Corollary 4.1 of \cite{BSW}.

Let $S_i$ and $P_i$ be the simple and projective modules at vertex $i$, respectively.  By attaching the correct weights on the usual Koszul resolution to make it equivariant, we get that the projective resolutions of the simple modules are of the form:
$$\rightarrow \bigopls_{1 \leq x < y \leq n} P_{k+a_x+a_y} \rightarrow \bigopls_{1 \leq x \leq n} P_{k+a_x} \rightarrow P_k
\rightarrow S_k \rightarrow 0.$$
$$ 0 \rightarrow P_{k+\sum a_i} \rightarrow \bigopls_{1 \leq x \leq n} P_{k-a_x+\sum a_i} \rightarrow
\bigopls_{1 \leq x < y \leq n} P_{k-a_x-a_y+\sum a_i}
\rightarrow \cdots $$
Note that the global dimension of $A$ is $n$, and that $A$ is positively graded and Noetherian. Now, Proposition \ref{prop1} yields the following.
\begin{Cor}
Let $e=e_k$ be a primitive idempotent of $A=S \rtimes \mu_m$ as above.   If $\sum_{i \in I} a_i \neq 0 \pmod{m}$ for all non-empty subsets $I \subseteq \Z/m$, then ${\rm gl.dim} \; \Gamma \leq 2n-1$.
\end{Cor}
For some particular examples, we could take weights $a_i=1$ for $n<m$ so we obtain an algebra $\Gamma$ that can be interpreted as a noncommutative resolution of the affine cone
of the $m^{th}$ Veronese embedding of $\mathbb{P}^{n-1}$.

\subsection{Skew group algebras of dimension two}
Let $G$ be a finite group and let $V$ be a two dimensional representation of $V$ via the map $\rho: G \rightarrow \GL(V)$.
Let $\overline{A}=k[V] \rtimes G$.  It is shown in Theorem 5.6 \cite{RVDB} that $\overline{A}$ is Morita equivalent to a basic algebra $A$ with
quiver given by the McKay graph.  The McKay graph is defined by letting $Q_0=\{1,\ldots,n\}$ be indexed by the set $\{ W_1,\ldots,W_n \}$ of irreducible representations of $G$.
The number of arrows from $i$ to $j$ is given by the dimension of the vector space $\Hom_G(W_i, W_j \otimes V)$.
As studied in \cite{RVDB}, this quiver is a finite translation quiver with translation $\tau(i) = j$ where $W_i \otimes \wedge^2 V \simeq W_j$.
We also know that $A$ is Noetherian and Koszul and has global dimension two.  The projective resolutions of the simples are given by
$$0\rightarrow P_{\tau(i)} \rightarrow \bigopls_{i \rightarrow j} P_j \rightarrow P_i \rightarrow S_i \rightarrow 0.$$  We make the following observation which is well known to experts.  We obtain that if there are no loops or $\tau$ loops at the vertex corresponding to $W_i$ then the subalgebra $\Gamma$ without that vertex has global dimension 3.

\begin{Cor} \label{cor8.2}Let $A$ be a skew group algebra of dimension two as above.
Let $e$ be the primitive idempotent corresponding to the irreducible representation $W_i$.  If $\Hom_G(W_i,W_i\otimes V)=0$ and
$W_i \otimes \wedge^2 V \not \cong W_i$ then $\rm{gl.dim} \Gamma = 3$.
\end{Cor}
\begin{proof} It follows from Proposition~\ref{prop1} that  $\rm{gl.dim} \Gamma \leq 3$.  Let $W_m = \wedge^2V^* \otimes W_i$ so $m = \tau^{-1}(i) \neq i$.
We have the resolutions
 $$0\rightarrow P_{i} \rightarrow \bigopls_{m \rightarrow j} P_j \rightarrow P_m \rightarrow S_m \rightarrow 0.$$
$$0\rightarrow P_{\tau(i)} \rightarrow \bigopls_{i \rightarrow \ell} P_\ell \rightarrow P_i \rightarrow S_i \rightarrow 0.$$
Note that $\Ext^j_A(S_m, S_i)=0$ for $j = 0, 1$ and $\Ext^j_A(S_i, S_i)=0$ for $j = 1,2$. Now we apply the functor $F$ and combine the resulting sequences to obtain the projective resolution
$$0\rightarrow F(P_{\tau(i)}) \rightarrow \bigopls_{i \rightarrow \ell} F(P_\ell) \rightarrow \bigopls_{m \rightarrow j} F(P_j) \rightarrow F(P_m) \rightarrow F(S_m) \rightarrow 0.$$
Note that the maps are all radical maps so this is a minimal resolution of $F(S_m)$
So we must have that  $\rm{gl.dim} \Gamma = 3$.
\end{proof}
For a more specific example, let $r>1$ and
let $G$ be the dihedral group
$\langle \s,\t \st \s^r = \t^2 = 1, \s\t=\t\s^{-1} \rangle$.
The representation $V$ is defined
by
\[ \s =
\begin{pmatrix}
 \z & 0 \\
  0 & \z^{-1}
\end{pmatrix} \ \ , \ \
  \t =
\begin{pmatrix}
  0 & 1 \\
  1 & 0
\end{pmatrix}
\]
where $\z$ is a primitive $r$-th root of unity.
Now we can let $e$ be the primitive idempotent corresponding to any of the one dimensional representations.  These are the trivial representation and
$\wedge^2 V$ when $r$ is odd, and there are four one dimensional representations when $r$ is even.  For more details, see Section 8.2 of \cite{CHI} where this example is referred to as type $BL$ for $r$ odd and type $B$ for $r$ even and in Example 5.1 of \cite{BSW} the dihedral group of order $8$, that is, $r=4$ is treated.

For another specific example, where the group does not contain any pseudo reflections, we can let $G$ be the subgroup of $\GL(V)$ generated by
$$ \begin{pmatrix}
\zeta_{6} & 0 \\ 0 & \zeta_{6} \end{pmatrix},
\begin{pmatrix}
0 & -\zeta_8 \\
-\zeta_8^3 & 0 \end{pmatrix},
\begin{pmatrix}
-\zeta_8 & 0 \\
0 & -\zeta_8^3  \end{pmatrix}$$
where $\zeta_n$ is a primitive $n^{th}$ root of unity.
All possible primitive idempotents will satisfy the conditions of Corollary \ref{cor8.2} so  $\rm{gl.dim} \Gamma = 3$.

For one last example, we generalize the above corollary to the case of more than one idempotent and we omit the proof which also follows from Proposition \ref{prop1}.
\begin{Cor} Let $A$ be a skew group algebra of dimension two as above.
Let $e$ be the idempotent corresponding to a direct summand $W$ of $\bigopls_{i=1}^n W_i$ and let $\Gamma=(1-e)A(1-e)$.
If $\Hom_G(W,W\otimes V)=0$ and
$\Hom(W,W \otimes \wedge^2 V)=0$ then $\rm{gl.dim} \Gamma \le 3.$
\end{Cor}
It is interesting to compare this with Theorem 2.10 of \cite{IW} where $e$ corresponds to the set of special Cohen-Macaulay modules over $k[V]^G$.
\bigskip

\noindent{\emph{Acknowledgment}.} The authors are supported by NSERC while the second author is also supported in part by AARMS.  The first author would like to thank Michael Wemyss for helpful comments. Both authors would like to express their gratitude to an anonymous referee for pointing out Lemma \ref{MinResolution2}, which made Theorem \ref{MainTheo} much stronger.

\end{document}